\documentclass[12pt,a4paper]{elsarticle}
\usepackage{amsmath,amsthm,amsfonts,amssymb,mathrsfs}

\newtheorem{theorem}{Theorem}
\newtheorem{lemma}[theorem]{Lemma}

\newtheorem{prop}[theorem]{Proposition}
\newtheorem{cor}[theorem]{Corollary}

\theoremstyle{remark}
\newtheorem{rem}[theorem]{Remark}

\numberwithin{equation}{section}
\numberwithin{theorem}{section}

\def \RR {\mathbb{R}}
\def \NN {\mathbb{N}}

\def \veps {\varepsilon}

\def\({\left(}
\def\){\right)}
\def\ls{\lesssim}

\begin{document}
\begin{frontmatter}

\title{Pointwise estimates\\ for solutions of  fractal Burgers equation\tnoteref{grant}}
\tnotetext[grant]{The research was partially supported by grant MNiSW IP2012 018472 and by NTU Tier 1 Grant RG19/12.}

\author[1]{Tomasz Jakubowski\corref{cor}}
\author[2,3]{Grzegorz Serafin}
\cortext[cor]{Corresponding author}

\address[1]{
Wroc{\l}aw University of Technology\\
ul. Wybrze{\.z}e Wyspia{\'n}skiego 27, 
50-370\\
Wroc{\l}aw, Poland\\
E-mail: tomasz.jakubowski@pwr.edu.pl\\\ 
}

\address[2]{
Wroc{\l}aw University of Technology\\
ul. Wybrze{\.z}e Wyspia{\'n}skiego 27, 
50-370\\
Wroc{\l}aw, Poland\\
E-mail: grzegorz.serafin@pwr.edu.pl\\\ 
}

\address[3]{
Division of Mathematical Sciences\\
School of Physical and Mathematical Sciences\\
Nanyang Technological University\\
21 Nanyang Link\\
Singapore 637371\\
grzegorz.serafin@pmail.ntu.edu.sg
}

\begin{abstract}
In this paper, we provide two-sided estimates and uniform asymptotics for  the solution of $d$-dimensional critical fractal Burgers equation $u_t-\Delta^{\alpha/2}u+b\cdot \nabla\(u|u|^q\)=0$, $\alpha\in(1,2)$, $b\in\RR^d$ for $q = (\alpha-1)/d$ and $u_0 \in L^1(\RR^d)$. We consider also $q > (\alpha-1)/d$ under additional condition $u_0 \in L^\infty(\RR^d)$. In both cases we assume $u_0\geq0$, which implies that the solution is non-negative. The estimates are given in the terms of the function $P_t u_0$, where $P_t$ is the stable semigroup operator.
\end{abstract}

\begin{keyword}
generalized Burgers equation, fractional Laplacian, estimates of solutions, asymptotics of solutions
\MSC[2010] 35B40, 35K55, 35S10
\end{keyword}

\end{frontmatter}

\section{Introduction}
Let $d\in\NN$, $\alpha\in(1,2)$ and $q_0 = (\alpha-1)/d$. The goal of the paper is to describe estimates and asymptotics of  solutions of the  fractal Burgers equation
\begin{equation}\label{eq:problem}
\left\{\begin{array}{l}u_t-\Delta^{\alpha/2}u+b\cdot \nabla\(u|u|^q\)=0,\\u(0,x)=u_0(x),\end{array}\right.
\end{equation}
where $q\ge q_0$ and $b\in\RR^d$ is a constant vector. We assume that $u_0\in L^1$  and $u_0\ge0$, cf. (\ref{con1}), (\ref{con2}). Then, the solution $u(t,x)$ is also non-negative and the absolute value in $(\ref{eq:problem})$ may be omitted. Furthermore, the pseudo-differential operator $\Delta^{\alpha/2}$  is the fractional Laplacian defined by the Fourier transform
$$
\widehat{\Delta^{\alpha/2} \phi}(\xi)=-|\xi|^{\alpha}\widehat\phi(\xi),\ \ \ \ \phi\in C_c^\infty(\RR^d).
$$
We denote the heat kernel related to this operator by $p(t,x)$. It is the fundamental solution of 
\begin{equation}\label{eq:hk}v_t=\Delta^{\alpha/2}v.\end{equation}
The corresponding semigroup operator  $P_t$ is given by
$$
P_t f(x) = \int_{\RR^d} p(t,x-y) f(y) dy.
$$
Linear and nonlinear gradient perturbations of fractional Laplacian have been intensely studied in recent years, e.g. \cite{J, JSz, Sz, BJ2, MM, ChKS2, S, CC, CV}. Equation (\ref{eq:problem})  was recently investigated in \cite{BFW, BKW1, BKW2, BK} for various values of $q$ and initial conditions $u_0$. For $d=1$, the case $q=2$ is of particular interest (see e.g. \cite{KMX, AIK, KNS, WW}) because it is a natural counterpart of the classical Burgers equation.
In \cite{BKW1} the authors studied the solution of \eqref{eq:problem} for $q=q_0$ and $u_0 = M \delta_0$, where $\delta_0$ is the Dirac measure at $0$ and $M>0$ is some constant. They showed  the existence of the solution $U_M(t,x)$ and its basic properties. In \cite{BK} pointwise estimates of $U_M(t,x)$ where derived for small values \mbox{of $M$}. More precisely, it was proved that for sufficiently small $M$,
$$
0 \le U_M(t,x) \le c \,p(t,x), \qquad t>0, x \in \RR^d,
$$  
for some constant $c>0$. This result was improved in the recent paper \cite{JS}. The authors showed that for every $M>0$, there is a constant $c>0$ such that the following estimates hold
$$
c^{-1} p(t,x)\le  U_M(t,x) \le c\, p(t,x), \qquad t>0, x \in \RR^d.
$$  
The aim of this paper it to obtain  similar results for $u_0$ satisfying either of the following conditions, which depend on the value of $q$:
\begin{align}
	\bullet\ & u_0\in L^1(\RR^d),\  u_0\ge0,& & \mbox{for $q=q_0$}, \label{con1} \\
	 \bullet\  & u_0\in L^1(\RR^d)\cap L^\infty(\RR^d),\ u_0\ge0,&&  \mbox{for $q>q_0$}. \label{con2}
\end{align}
Additionally, we  assume throughout the paper that  $\|u_0\|_1=M>0$. The value $q_0$ is called a critical exponent. In this case linear and non-linear operators  are balanced, whereas  the fractional Laplacian is dominating for $q>q_0$. More precisely, the large time behaviour of the solution for $q>q_0$ coincides with  behaviour of  the solution $\(P_tu_0\)(x)$ of (\ref{eq:hk}) (\cite{BKW2},  Theorem 4.1)
\begin{equation}\label{eq:u-P}
\lim_{t\rightarrow\infty}t^{n(1-1/p)/\alpha}\left\|u(t,\cdot)-\(P_tu_0\)(\cdot)\right\|_p=0,\ \ \ \ \text {for each } p\in[1,\infty].
\end{equation}  
On the other hand, for $q=q_0$ the large time behaviour of the solution of (\ref{eq:problem}) is governed by the self-similar fundamental solution $
U_M(x,t)$:
\begin{equation}\label{eq:u-U}
\lim_{t\rightarrow\infty}t^{n(1-1/p)/\alpha}\left\|u(t,\cdot)-U_M(t,\cdot)\right\|_p=0,\ \ \ \ \text {for each } p\in[1,\infty],
\end{equation} 
whith $M=\left\|u_0\right\|_1$  (\cite{BKW1}, Theorem 2.2). Analogous results hold for $\alpha=2$ (\cite{EZ}, see also \cite{EVZ1}, \cite{EVZ2} for related problems). In the paper, we improve \eqref{eq:u-P} and provide some other asymptotics. However, our main result is as follows. 

\begin{theorem}\label{thm:main}
Let one of the conditions \eqref{con1}, \eqref{con2} holds. Then, the solution $u(t,x)$ of (\ref{eq:problem}) satisfies
$$\frac1C\(P_tu_0\)(x)\le u(t,x)\le C\(P_tu_0\)(x)$$
for some $C=C(d,\alpha,u_0)>1$. 
\end{theorem}

Since we do not know the exact behaviour of $u_0$, we cannot give the precise estimates of $P_t u_0$. For example, for $u_0(x) = \frac{1}{1+|x|^{d+\gamma}}$, where $\gamma \in (0,\alpha]$, $P_1u_0(x)  \approx \frac{1}{1+|x|^{d+\gamma}}$ (see, for example, \cite{ChKS1}). In order to prove Theorem \ref{thm:main}, we introduce a function $u^*(t,x) = t^{d/\alpha} u(t,t^{1/\alpha}x)$, which is very convenient to deal with. In particular, the estimates of the $L^p$ norms of $u^*(t,\cdot)$ does not depend on $t$. 
It is  worth mentioning that the methods used to prove Theorem \ref{thm:main} may also be applied in the case when $u_0 = M\delta_0$ and improve the techniques used in the paper \cite{JS}.  

The paper is organized as follows. In Section 2, we collect some properties of $p(t,x)$ and introduce Duhamel formula. In Section 3, we show some basic asymptotics of the solution $u(t,x)$ as $t \to 0$ or $|x| \to \infty$. Section 4 is devoted to prove Theorem \ref{thm:main}. Finally, in Section 5, we give the precise description of asymptotic behaviour of the function $u(t,x)$.

\section{Preliminaries}

\subsection{Notation}
For two positive functions $f,g$ we denote $f\ls g$ whenever there exists a constant $c>1$ such that $f(x)<cg(x)$ for every argument $x$.  If $f\ls g$ and $g\ls f$ we write $f\approx g$. Enumerated constants denoted by capital letters do not change in the whole paper while constant denoted by small letters may change from lemma to lemma. By $|\cdot|$ we denote the Euclidean norm in $\RR$ and $\RR^d$.

\subsection{Properties of $p(t,x)$}
 The function $p(t,x)$ was introduced as a fundamental solution of (\ref{eq:hk}). We recall that it is a kernel of the stable semigroup
$ \(P_t f\)(x) =\int_{\RR^d}p(t,x,w)f(w)dw$, where $p(t,x,w)=p(t,x-w)$.
It  may be also given by the inverse Fourier transform 
\begin{equation*}
	p(t,x) = (2\pi)^{-d} \int_{\RR^d} e^{-i x \cdot \xi} e^{-t |\xi|^{\alpha}}d\xi, \qquad t>0\,,\; x \in \RR^d\,.
\end{equation*}
As a consequence, the following scaling property holds
\begin{equation}
\label{eq:scalingp}
p(t,x)=\lambda^{d/\alpha}p(\lambda t,\lambda^{1/\alpha}x), \hspace{15mm}\lambda>0.
\end{equation}
Furthermore, estimates of both: the function and its gradient are well-known (see e.g. \cite{BJ}) and can be expressed by
\begin{align}\label{est:p}
p(t,x,y)&\approx \frac t{\(t^{1/\alpha}+|y-x|\)^{d+\alpha}},\\
\label{est:gradp}
|\nabla_yp(t,x,y)|&\approx \frac {t\, |y-x|}{\(t^{1/\alpha}+|y-x|\)^{d+2+\alpha}}.
\end{align}
In particular, we have 
\begin{equation}\label{est:gradpp}
|b\cdot\nabla_yp(t,x,y)|\ls   t^{-1/\alpha}p(t,x,y),
\end{equation}
where $b\in\RR^d$ is a constant vector.
\subsection{Duhamel formula}
One of the main tools we use in this paper is the following Duhamel formula
\begin{equation}
u(t,x) = \(P_t u_0\)(x) + \int_0^t \int_{\RR^d} p(t-s,x,z) b \cdot \nabla_z [u(s,z)]^{q+1}\,dz\,ds.
\label{eq:Duhamel1}
\end{equation}
Here, we used the fact that $u(t,x)$ is non-negative. Integrating by parts, we get
\begin{equation}
u(t,x) =  \(P_t u_0\)(x) - \int_0^t \int_{\RR^d} b \cdot \nabla_z p(t-s,x,z)  [u(s,z)]^{q+1}\,dz\,ds.
\label{eq:Duhamel2}
\end{equation}
Let us denote
\begin{equation*}
u^*(t,x)= t^{d/\alpha}u(t,t^{1/\alpha}x).
\end{equation*}
We note that $u^*(t,x)= u^t(1,x)$, where $u^\lambda(t,x) = \lambda^{d/\alpha}u(\lambda t, \lambda^{1/\alpha} x)$ is the rescaled solution, cf. (\ref{eq:scalingp}). Although the function $u^*(t,x)$ depends on time, it plays a similar role as $U_M(1,x)$ in \cite{JS}. 

Let us observe that under (\ref{con1}) or (\ref{con2}), we have
$$
u(t,x)^{q+1} = u(t,x)^{q_0+1} u(t,x)^{q-q_0} \le c\, u(t,x)^{q_0+1},  
$$
where $c=1$ in the case $q=q_0$ and $c=\sup_{t>0}\|u_0(\cdot)\|^{q-q_0}_\infty$ in the case $q>q_0$ (see formula 3.7 in \cite{BKW2}).
Now, by scaling property of $p(t,x)$ and some substitutions in the integrals, we have
\begin{align}\label{eq:intduh}
t^{d/\alpha} \int_0^t \int_{\RR^d}& b \cdot \nabla_z p(t-s,t^{1/\alpha}x,z)  [u(s,z)]^{q_0+1}\,dz\,ds\\\nonumber
&=t^{-1/\alpha} \int_0^t \int_{\RR^d} b \cdot \nabla_z p(1-s/t,x,t^{-1/\alpha}z)  [u(s,z)]^{q_0+1}\,dz\,ds\\\nonumber
&=t^{1-1/\alpha} \int_0^1 \int_{\RR^d} b \cdot \nabla_z p(1-v,x,t^{-1/\alpha}z)  [u(vt,z)]^{q_0+1}\,dz\,dv\\\nonumber
&=t^{(d+\alpha-1)/\alpha} \int_0^1 \int_{\RR^d}v^{d/\alpha}\, b \cdot \nabla_w p(1-v,x,v^{1/\alpha}w)  [u(vt,(vt)^{1/\alpha}w)]^{q_0+1}\,dw\,dv\\\nonumber
&=\alpha \int_0^1 \int_{\RR^d} b \cdot \nabla_w p(1-r^\alpha,x,rw) (rt^{1/\alpha})^{d(q_0+1)} [u(r^\alpha t,rt^{1/\alpha}w)]^{q_0+1}\,dw\,dr.
\end{align}
Finally, we get in both cases (\ref{con1}) and (\ref{con2})
\begin{align}
u^*(t,x) & \le\( P^*_t u_0\)(x) + c\, t^{d/\alpha} \int_0^t \int_{\RR^d} |b \cdot \nabla_w p(t-s,t^{1/\alpha}x,z)|  [u(s,z)]^{q_0+1}\,dz\,ds, \label{eq:duhamel*est1} \\
&=  \( P^*_t u_0\)(x) + c\,\alpha\int_0^1 \int_{\RR^d} |b \cdot \nabla_w p(1-r^\alpha,x,rw)|  [u^*(r^\alpha t,w)]^{q_0+1}\,dw\,dr, \nonumber \\
&\le  \( P^*_t u_0\)(x) + C_1\int_0^1 \int_{\RR^d} \frac{p(1-r^\alpha,x,rw)}{(1-r^\alpha)^{1/\alpha}+|x-rw|}  [u^*(r^\alpha t,w)]^{q_0+1}\,dw\,dr, \label{eq:duhamel*est2}
\end{align}
where $C_1=C_1(d,\alpha,u_0)$ and
$$ \( P^*_t u_0\)(x)=t^{d/\alpha}\(P_t u_0\)(t^{1/\alpha}x).$$
We note that $P^*_t$ is not a semigroup, we use this notation by the similarity to the definition of $u^*$.
\section{Properties of $u^*$}
The function $u^*(t,x)$ possesses some convenient properties which make it very useful to deal with. First of them is a uniform upper bound of every $L^p$-norm.

\begin{lemma}There exists $C=C(d)>0$ such that 
\begin{align}\label{est:pnormu*}
\left\|u^*(t,\cdot)\right\|_p<C\left\|u_0\right\|_1,\ \ \ t>0,\ p\in[1,\infty]. \end{align}
\end{lemma}
\begin{proof}We base on the formula 3.14 in \cite{BKW2}, which implies that for every $p\in[1,\infty]$ there exists $C_{d,p}>1$ such that
\begin{equation}\label{est:pnormu}
\left\|u(t,x)\right\|_p<C_{d,p}\left\|u_0\right\|_1t^{-d(1-1/p)/\alpha}.
\end{equation}
This directly gives us   (\ref{est:pnormu*})  for $p=\infty$. For $p=1$, it is enough to substitute $x=t^{1/\alpha}w$ when computing the $L^1$ norm. For $p\in(1,\infty)$,  we use the elementary interpolation inequality and get
\begin{align*}
\left\|u^*(t,\cdot)\right\|_p\le\left\|u^*(t,\cdot)\right\|_\infty^{1-1/p} \left\|u^*(t,\cdot)\right\|_1^{1/p}\le C_{d,1}C_{d,\infty}\left\|u_0\right\|_1,
\end{align*}
which ends the proof.
\end{proof}
The next two propositions show that the function $u^*(t,x)$ decays uniformly as $t$ tends to zero or infinity.

\begin{prop}\label{lem:limt0}
Assume (\ref{con1}) or (\ref{con2}) holds. Then, we have
$$\lim_{t\rightarrow0}\left\|u^*\(t,\cdot\)\right\|_\infty=0.$$
\end{prop}
\begin{proof}
Let $0<\varepsilon<1/2$. There exists $R>0$ such that
$\int_{|u_0|>R}u_0(w)dw<\varepsilon$. 
Then, using estimates (\ref{est:p}) of $p(t,x)$, we get
\begin{align*}
\( P^*_t u_0\)\(x\)&= \int_{|u_0|>R}p(1,x,wt^{-1/\alpha})u_0(w)dw+t^{d/\alpha}\int_{|u_0|\le R}p(t,xt^{1/\alpha},w)u_0(w)dw\\
&\ls \varepsilon+R\,t^{d/\alpha},
\end{align*}
which is small enough for $t$ close to zero. Now, by (\ref{est:gradpp}), the integral in (\ref{eq:duhamel*est1}) may be estimated by 
\begin{align*}
&t^{d/\alpha}\int_0^t \int_{\RR^d} \left|b \cdot \nabla_z p(t-s,xt^{1/\alpha},z) \right| [u(s,z)]^{q_0+1}\,dz\,ds\\
&\ls \int_0^t \int_{\RR^d} t^{d/\alpha}(t-s)^{-1/\alpha}p(t-s,xt^{1/\alpha},z)  [u(s,z)]^{q_0+1}\,dz\,ds\\
&=\int_0^{(1-\varepsilon) t}+ \int_{(1-\varepsilon)t}^t=I_1(t,x)+I_2(t,x).
\end{align*}
We start with estimating $I_2(t,x)$. By (\ref{est:pnormu}) with $p=\infty$, we have
\begin{align*}
I_2(t,x)&\ls t^{d/\alpha}  \int_{(1-\varepsilon)t}^{ t} (t-s)^{-1/\alpha}\int_{\RR^d} p(t-s,xt^{1/\alpha},z) ((1-\varepsilon) t)^{-d(q_0+1)/\alpha} \,dz\,ds\\
&\ls t^{-(\alpha-1)/\alpha}\int_{0}^{\varepsilon t} s^{-1/\alpha}ds=\frac{\varepsilon^{1-1/\alpha}}{1-1/\alpha}.
\end{align*}
Furthermore, the formulae (\ref{est:gradpp}) and (\ref{est:pnormu})  with $p=\infty$ imply
\begin{align*}
I_1(t,x)\ls \varepsilon^{-1/\alpha}t^{(d-1)/\alpha}\int_0^{(1-\varepsilon)t}s^{-(\alpha-1)/\alpha} \int_{\RR^d}  p(t-s,xt^{1/\alpha},z)u(s,z)\,dz\,ds.
\end{align*}
We take $\widetilde{R}>0$ such that $\int_{|u_0(x)|>\widetilde{R}}|u_0(x)|ds<\varepsilon^{(d+2)/\alpha}$. Let $v(t,x)$ be a solution of the problem (\ref{eq:problem}) with initial condition $v(0,x)=\mathbf{1}_{\{|u_0(x)|<\widetilde{R}\}}u_0(x)$. Thus,  for every $t>0$,  we obtain (\cite[Lemma 3.1]{BKW1})
\begin{align*}
\left\|v(t,\cdot)-u(t,\cdot)\right\|_1&\le\left\|v(0,\cdot)-u(0,\cdot)\right\|_1<\varepsilon^{(d+2)/\alpha},\\
\left\|v(t,\cdot)\right\|_\infty&\le\left\|v(0,\cdot)\right\|_\infty\le \widetilde{R}.
\end{align*}
Consequently, 
\begin{align*}
I_1(t,x)&\ls\varepsilon^{-1/\alpha}t^{(d-1)/\alpha}\int_0^{(1-\varepsilon)t}s^{-(\alpha-1)/\alpha} \int_{\RR^d}  p(t-s,xt^{1/\alpha},z)v(s,z)\,dz\,ds\\
&+\varepsilon^{-1/\alpha}t^{(d-1)/\alpha}\int_0^{(1-\varepsilon)t}s^{-(\alpha-1)/\alpha} \int_{\RR^d}  p(t-s,t^{1/\alpha}x,z)|u(s,z)-v(s,z)|\,dz\,ds.\\
&\ls \widetilde{R}\varepsilon^{-1/\alpha}t^{d/\alpha} +\varepsilon^{-1/\alpha}t^{(d-1)/\alpha}\int_0^{(1-\veps)t} s^{-(\alpha-1)/\alpha} \int_{\RR^d} (\veps t)^{-d/\alpha}   |u(s,z)-v(s,z)|\,dz\,ds.\\
&\ls \widetilde{R}\varepsilon^{-1/\alpha}t^{d/\alpha} +\varepsilon^{1/\alpha}.
\end{align*}
Therefore, $\lim\limits_{t \to 0} \|u^*(t,\cdot)\|_\infty \le c \veps^{1/\alpha}$ for all $\veps \in (0,1/2)$, which ends the proof.
\end{proof}

\begin{prop}\label{lem:limxinfty}
Assume (\ref{con1}) or (\ref{con2}) holds. Then, we have
$$\lim_{|x|\rightarrow\infty}\|u^*(\cdot,x)\|_\infty=0.$$
\end{prop}
\begin{proof}
Let $0<\varepsilon<1/2$. By Proposition \ref{lem:limt0}, there exists $t_0>0$ such that $\|u^*(t,\cdot)\|_\infty<\varepsilon$ for $t\le t_0$. Therefore, we have to consider only  $t>t_0$. We will show that both terms in (\ref{eq:duhamel*est2}) tends uniformly to zero as $t\rightarrow0$. Since $u_0\in L^1$, there is a radius $R>0$ such that $\int_{|x|>R}u_0(x)dx<\varepsilon$. Then, by (\ref{est:p}), we get  for $|x|>R/t_0^\alpha$ 
\begin{align*}
\( P^*_t u_0\)(x)&=\int_{|w|>R}+\int_{|w|\leq R}p(1,x,t^{-1/\alpha}w)u_0(w)dw\\[10pt]
&\ls \varepsilon+\frac{\left\|u_0\right\|_1}{\(|x|-R/t^{1/\alpha}_0\)^{d+\alpha}},
\end{align*} 
which is small   for large $|x|$.  In order to estimate the integral in (\ref{eq:duhamel*est2}) we divide it as follows
\begin{align*}
&\int_0^1\int_{\RR^d}\frac{p(1-r^\alpha,x,rw)}{(1-r^\alpha)^{1/\alpha}+|x-wr|}[u^*(r^\alpha t,w)]^{q_0+1}dw\,dr\\
&=\int_0^\varepsilon+\int_\varepsilon^{(1-\varepsilon)^{1/\alpha}}+\int_{(1-\varepsilon)^{1/\alpha}}^1:=I_1+I_2+I_3.
\end{align*}
 Applying (\ref{est:p}) and (\ref{est:pnormu*}) for $p=1+q_0$, we obtain
\begin{align*}
I_1&\ls (1-\varepsilon^\alpha)^{-(d+1)/\alpha}\int_0^\varepsilon\int_{\RR^d}[u^*(r^\alpha t,w)]^{q_0+1}dw\,dr\\
&\ls \varepsilon\|u_0\|_1^{q_0+1}.
\end{align*}
Next, by (\ref{est:pnormu*}) for $p=\infty$, 
\begin{align*}
I_3&\ls\|u_0\|_1^{q_0+1}\int_{(1-\varepsilon)^{1/\alpha}}^1(1-r^\alpha)^{-1/\alpha}\int_{\RR^d}p(1-r^\alpha,x,rw)dw\,dr\\
&=\|u_0\|_1^{q_0+1}\int_{(1-\varepsilon)^{1/\alpha}}^1(1-r^\alpha)^{-1/\alpha}r^{-d}\,dr\\
&=\frac1\alpha\|u_0\|_1^{q_0+1}\int_0^{\varepsilon}s^{-1/\alpha}(1-s)^{1/\alpha-d-1}\,ds\\
&\ls \|u_0\|_1^{q_0+1} \varepsilon^{1-1/\alpha}.
\end{align*}
Now we are going to  deal with the integral $I_2$. By \cite[Lemma 3.10]{BKW1}, we have
\begin{equation*}
\lim_{R\rightarrow\infty}\sup_{t>T}\int_{|x|>R}u^*(t,x)dx=0, \qquad \mbox{for every $T>0$.}
\end{equation*}
Hence, there exists $R>0$ such that $\int_{|w|>R}u^*(s,w)dw<\varepsilon^{(d+1+\alpha)/\alpha}$ for every $s>\varepsilon^{1/\alpha} t_0$. Thus, for $|x|>R$, we get
 \begin{align*}
I_2&\ls  \|u_0\|_1^{q_0}\int_{(\varepsilon,(1-\varepsilon)^{1/\alpha})}\int_{|w|> R} (1-r^\alpha)^{-(d+1)/\alpha} u^*(r^\alpha t,w)dw\,dr\\
&+\int_{(\varepsilon,(1-\varepsilon)^{1/\alpha})}\int_{|w|\le R} \frac1{(x-rw)^{d+1+\alpha}} [u^*(r^\alpha t,w)]^{1+q_0}dw\,dr\\
&\ls  \|u_0\|_1^{q_0}\varepsilon^{-(d+1)/\alpha}\int_0^1\int_{|w|> R} u^*(r^\alpha t,w)dw\,dr
+\int_0^1\int_{\RR^d}\frac{[u^*(r^\alpha t,w)]^{1+q_0}}{(|x|-R)^{d+1+\alpha}}dw\,dr\\
&\ls \varepsilon\|u_0\|_1^{q_0}+\frac{\|u_0\|_1^{1+q_0}}{(|x|-R)^{d+1+\alpha}},
\end{align*}
which is small enough for sufficiently large $|x|$. This ends the proof.
\end{proof}
\begin{rem}
Using the same methodology, a noticeably simpler proof of Proposition 3.2 in \cite{JS} may be obtained.
\end{rem}

\section{Main results}
The main goal of this section is to prove Theorem  \ref{thm:main}. Additionally,  we present  some asymptotics of the function $u^*(t,x)$, which  play a crucial role in proof of the main theorem. Nevertheless, they are also interesting as separate results, which is discussed  in Section 5, where asymptotics of $u(t,x)$ are studied. 

To shorten notation, we denote for $\beta\in[0,1)$
\begin{align}\label{def:pb}
	h_\beta(r,x,w)=  r^{-\beta}(1-r^\alpha)^{-1/\alpha} p(1-r^\alpha,x,rw). 
\end{align}
The below-given technical lemma is intensively exploit  in proofs of Theorems \ref{thm:smallt} and \ref{thm:largex}.
\begin{lemma}\label{lem:2int}
Let $\beta\in(0,1)$ and $f:\RR^d\rightarrow[0,\infty)$, $g:[0,\infty)\times\RR^d\rightarrow[0,\infty)$ be such that integrals in \eqref{eq:intpP} and \eqref{eq:intppu} converge. There exist $C_2=C_2(d,\alpha, \beta)$ and $C_3=C_3(d,\alpha, \beta)$ such that
\begin{align}
(i)\ \ \int_0^1& \int_{\RR^d} h_\beta(r,x,w)  \( P^*_{r^\alpha t} f\)(w)\,dw\,dr=C_2\( P^*_{t}f\)(x), \label{eq:intpP}\\
(ii)\ \ \int_0^1& \int_{\RR^d} h_\beta(r,x,w)  \int_0^1   h_0(s,w,z)g(s^\alpha r^\alpha t,z) \,ds\,dw\,dr  \label{eq:intppu}\\
&\hspace{40mm}<C_3\int_0^1h_\beta(r,x,z)g(r^\alpha t,z)\,dr, \notag
\end{align}
where $t>0$, $z\in\RR^d$.
\end{lemma}
\begin{proof}
We note that for any $s,t,\beta,\gamma \in (0,\infty)$ and $x,z \in \RR^d$, by scaling property (\ref{eq:scalingp})  of  $p(t,x,y)$, we get 
\begin{align}
\int_{\RR^d} p(s,x,\beta w) p(t, \gamma w, z) dw &= \int_{\RR^d} \beta^{-d}p(\beta^{-\alpha}s,\beta^{-1}x,w) \gamma^{-d}p(\gamma^{-\alpha}t,  w, \gamma^{-1}z) dw \notag\\
&=  (\beta\gamma)^{-d}p(\beta^{-\alpha}s+\gamma^{-\alpha}t,\beta^{-1}x, \gamma^{-1}z) dw \notag\\
&= p(\gamma^{\alpha}s + \beta^{\alpha}t,\gamma x, \beta z).  \label{eq:chap4}
\end{align}
Since
$$
\(P^*_s f\)(w) = \int_{\RR^d} s^{d/\alpha} p(s, s^{1/\alpha}w,z)f(z)dz 
=\int_{\RR^d}  p(1, w, s^{-1/\alpha} z)f(z)dz,
$$
 (\ref{eq:chap4}) gives us
\begin{align*}
 \int_{\RR^d}   & p(1-r^\alpha,x,rw)  \( P^*_{r^\alpha t} u_0\)(w)\,dw \\
 &= \int_{\RR^d} \int_{\RR^d} p(1-r^\alpha,x,rw)  p(1, w, (rt^{1/\alpha})^{-1} z)u_0(z)dw\,dz\\
  &=  \int_{\RR^d} p(1,x,t^{-1/\alpha}z) u_0(z) \,dz  =\(P^*_{t} u_0\)(x).
\end{align*}
Thus,
\begin{align*}
\int_0^1 \int_{\RR^d} & h_\beta(r,x,w)  \( P^*_{r^\alpha t} f\)(w)\,dw\,dr\\
   &=\( P^*_{t} u_0\)(x)\int_0^1 r^{-\beta} (1-r^\alpha)^{-1/\alpha} dr=C_2\( P^*_{t} f\)(x),
\end{align*}
which proves $(i)$. Furthermore, by (\ref{eq:chap4}), we have
\begin{align*}
\int_{\RR^d}    p(1-r^\alpha,x,rw)    p(1-s^\alpha,w,sz) \,dw = p(1-(rs)^\alpha,x,rsz).
\end{align*}
Hence, substituting $s=v/r$ in the second line, we get
\begin{align*}
\int_0^1& \int_{\RR^d}  h_\beta(r,x,w)  \int_0^1  h_0(s,w,z)g(s^\alpha r^\alpha t,z) \,ds\,dw\,dr\\
&=\int_0^1 \int_0^1 r^{-\beta}  (1-r^\alpha)^{-1/\alpha}   (1-s^\alpha)^{-1/\alpha} p(1-(rs)^\alpha,x,rsz)g(s^\alpha r^\alpha t,z) \,ds\,dr\\
& =\int_0^1 \int_0^rr^{-\beta}  (1-r^\alpha)^{-1/\alpha}   (r^\alpha-v^\alpha)^{-1/\alpha}p(1-v^\alpha,x,vz)g(v^\alpha t,z)  \,dv\,dr\\
& =\int_0^1 p(1-v^\alpha,x,vz)g(v^\alpha t,z)\int_v^1r^{-\beta}  (1-r^\alpha)^{-1/\alpha}   (r^\alpha-v^\alpha)^{-1/\alpha}  \,dr\,dv.
\end{align*}
Using the estimate (\cite{JS}, Corollary 4.3)
$$\int_v^1 r^{-\beta}(1-r^\alpha)^{-1/\alpha}(r^\alpha-v^\alpha)^{-1/\alpha}    dr\ls v^{-\beta}(1-v)^{-1/\alpha},$$
we obtain the assertion $(ii)$.
\end{proof}
Theorems  \ref{thm:smallt} and \ref{thm:largex}  show that the  distance between $u^*(t,x)$  and $P^*_tu_0$ tends to zero as $t\rightarrow0$ or $|x|\rightarrow\infty$. To avoid repeating long integrals in the proofs of those theorems,  we rewrite  \eqref{eq:duhamel*est2} as
\begin{align}\label{eq:smallt}
u^*(t,x) \le P^*_t u_0(x) + I(t,x),
\end{align}
where
\begin{align}\label{eq1:smallt}
0 \le I(t,x) \le C_1 \int_0^1 \int_{\RR^d} \frac{p(1-r^\alpha,x,rw)}{(1-r^{\alpha})^{1/\alpha}+|x-rw|}  [u^*(r^\alpha t,w)]^{q_0+1}\,dw\,dr.
\end{align}

\begin{theorem}\label{thm:smallt} Assume \eqref{con1} or \eqref{con2} holds. We have
\begin{equation}
\lim_{t\rightarrow0}\left\|\frac{u^*(t,\cdot)}{\(P^*_t f\)(\cdot)} -1 \right\|_\infty=0.
\end{equation}
\end{theorem}

\begin{proof}
First, we  estimate the integral $I(t,x)$ from \eqref{eq1:smallt} as follows
$$
0 \le I(t,x)\le C_1 \int_0^1 \int_{\RR^d}  (1-r^\alpha)^{-1/\alpha}p(1-r^\alpha,x,rw)  [u^*(r^\alpha t,w)]^{q_0+1}\,dw\,dr.
$$
Let $0<\eta, \beta<1$. By Proposition \ref{lem:limt0}, we may choose $t_0$ such that 
\begin{align}\label{eq:aux0}
u^*(t,x)<\(\frac{\eta}{C_1(C_2\vee C_3)}\)^{1/q_0}, \qquad \mbox{for $t<t_0$ and $x \in \RR^d$,}
\end{align}
where $C_2$ and $C_3$ are the constants from Lemma \ref{lem:2int}. We will show that 
\begin{align}
I(t,x)\leq  \frac{\eta}{1-\eta}\( P^*_t u_0\)(x), \qquad t<t_0, \; x \in \RR^d. \label{est:Itx1}
\end{align}
Let $t<t_0$. 
Then, using notation introduced in (\ref{def:pb}),
 \begin{align}\label{eq:aux1}
I(t,x)\le  \frac{\eta}{C_2\vee C_3}\int_0^1 \int_{\RR^d}  h_\beta(r,x,w)  u^*(r^\alpha t,w)\,dw\,dr := J(t,x).
\end{align}
We note that by \eqref{eq:smallt}, we have
\begin{equation}\label{eq:aux2}
u^*(r^\alpha t,w)\leq \( P^*_{r^\alpha t} u_0\)(w)+C_1\int_0^1 \int_{\RR^d}  h_0(s,w,z)  [u^*(s^\alpha r^\alpha t,z)]^{q_0+1}\,dz\,ds.
\end{equation}
We apply \eqref{eq:aux2} to (\ref{eq:aux1}) and, by Lemma \ref{lem:2int} and \eqref{eq:aux0}, we get
\begin{align}
J(t,x) &\leq  \eta\( P^*_t u_0\)(x) + C_1 \eta\int_0^1 \int_{\RR^d}  h_\beta(r,x,z)  [u^*(r^\alpha t,z)]^{q_0+1}\,dz\,dr, \notag\\
&\leq  \eta\( P^*_t u_0\)(x) + \eta J(t,x). \label{eq:aux3}
\end{align}
Hence, $(1-\eta) J(t,x) \le \eta\( P^*_t u_0\)(x)$ and, by (\ref{eq:aux1}), we get \eqref{est:Itx1}. 

Consequently, for $t<t_0$ and $x\in \RR^d$, $\left|u^*(t,x) -P^*_t u_0(x)\right| \le \frac{\eta}{1-\eta} P^*_t u_0(x) $, which is equivalent to
$$ \left\|\frac{u^*(t,\cdot)}{\(P^*_t u_0\)(\cdot)} -1 \right\|_\infty \le \frac{\eta}{1-\eta},\ \ \ 0<t<t_0.$$
The proof is completed.
\end{proof}

Using  a similar method we get the asymptotics of $u^*(t,x)$ for $|x| \to \infty$.

\begin{theorem}\label{thm:largex}
Assume \eqref{con1} or \eqref{con2} hold. We have
\begin{equation}
\lim_{|x|\rightarrow\infty}\sup_{t>0}\left|1-\frac{u^*(t,x)}{\(P^*_t f\)(x)}\right|=0.
\end{equation}
\end{theorem}

\begin{proof}
Let $0<\eta,\beta<1$. By Proposition \ref{lem:limxinfty}  we may choose $R>0$ such that $u^*(t,x)<\(\frac{\eta}{C_1(C_2\vee C_3)}\)^{1/q_0}$ for $|x|>R$ and $t>0$. We divide  the integral $I(t,x)$ from \eqref{eq1:smallt} into $\int_0^1\int_{|w|\le R}+\int_0^1\int_{|w|> R}$  and estimate  it as follows 
 \begin{align}\label{eq:duhineq1}
I(t,x)  &\leq  C_1\int_0^1 \int_{|w|<R} \frac{p(1-r^\alpha,x,rw)}{|x-rw|}  [u^*(r^\alpha t,w)]^{1+q_0}\,dw\,dr\\\nonumber
  &\ \ +   \frac{\eta }{C_2\vee C_3}\int_0^1 \int_{|w|>R}  h_\beta(r,x,w)  u^*(r^\alpha t,w)\,dw\,dr.
\end{align}
 Similarly, we  get
 \begin{align}\nonumber
u^*(r^\alpha t,w)\leq&\, P_{r^\alpha t}^*u_0(x)+  \int_0^1 \int_{|z|\le R}    h_0(s,w,z)  [u^*(s^\alpha r^\alpha t,z)]^{1+q_0}\,dz\,ds\\\label{eq:duhineq2}
  &+ \frac{\eta}{C_2\vee C_3}\int_0^1 \int_{|z|>R}   h_0(s,w,z)  u^*(s^\alpha r^\alpha t,z)\,dz\,ds.
\end{align}
First, we will estimate the last expression in (\ref{eq:duhineq1})
\begin{equation}
J(t,x) = \frac{\eta}{C_2\vee C_3}\int_0^1 \int_{|w|>R}  h_\beta(r,x,w)  u^*(r^\alpha t,w)\,dw\,dr. \label{eq:duhineq0}
\end{equation}
We put \eqref{eq:duhineq2} into \eqref{eq:duhineq0} and, by virtue of  Lemma \ref{lem:2int}, we  get
\begin{align*}
J(t,x) & \leq  \eta\( P^*_t u_0\)(x) +\eta \int_0^1 \int_{|w|<R} h_\beta(r,x,w)  [u^*(r^\alpha t,w)]^{1+q_0}\,dw\,dr + \eta J(t,x).
\end{align*}
Hence,
\begin{align}\label{eq:duhineq3}
J(t,x) &\leq  \frac\eta{1-\eta}\( P^*_t u_0\)(x) +\frac\eta{1-\eta} \int_0^1 \int_{|w|<R} h_\beta(r,x,w)  [u^*(r^\alpha t,w)]^{1+q_0}\,dw\,dr.
\end{align}
Then, for $|x|>(2R)\vee\frac1\eta$ by (\ref{eq:duhineq1}) and (\ref{eq:duhineq3}), we get

 \begin{align*}\nonumber
I(t,x) &\leq  c_1 \int_0^1 \int_{|w|<R} \(\frac{2}{|x|}\)^{d+\alpha+1}  [u^*(r^\alpha t,w)]^{1+q_0}\,dw\,dr+   \frac\eta{1-\eta}\( P^*_t u_0\)(x)\\\nonumber
    &\ \ +\frac\eta{1-\eta}c_1 \int_0^1 \int_{|w|<R} r^{-\beta}(1-r^\alpha)^{-1/\alpha}\(\frac{2}{|x|}\)^{d+\alpha} [u^*(r^\alpha t,w)]^{1+q_0}\,dw\,dr\\\nonumber
  &\le \frac\eta{1-\eta}\( P^*_t u_0\)(x)+\frac\eta{1-\eta}\frac {c_2}{|x|^{d+\alpha}},
\end{align*}
for some $c_2=c_2(d,\alpha,\beta,u_0)>0$. The next step is to prove that  $\( P^*_t u_0\)(x)\gtrsim \frac1{|x|^{d+\alpha}}$ for large $|x|$ and $t$ bounded away from zero. There is $r_0>0$ such that $\int_{|w|<r_0}u_0(w)dw>\|u_0\|_1/2$. Let $t_0 >0$. For $t>t_0$ and $x \in \RR^d$ we get
\begin{align}\nonumber
\( P^*_t u_0\)(x)&\ge \int_{B(0,r_0)} p(1,x,t^{-1/\alpha}w)u_0(w)dw\\\label{eq:P>p}
&\gtrsim \frac{1}{(1+|x|+r_0/t_0^{1/\alpha})^{d+\alpha}} \int_{|w|<r_0}u_0(w)dw\gtrsim \frac{\|u_0\|_1}{1 \vee |x|^{d+\alpha}}.
\end{align}
Combining all together, there exists $c_3=c_3(d,\alpha,\beta,u_0)>0$ such that
 \begin{align*}\nonumber
&\left| I(t,x)\right|\le c_3\frac\eta{1-\eta}\( P^*_t u_0\)(x)
\end{align*}
holds whenever $|x|>(2R)\vee\frac1\eta$ and $t>t_0$. Consequently
$$\left|1-\frac{u^*(t,x)}{\(P^*_t u_0\)(x)}\right|<c_3\frac{\eta}{1-\eta}.$$
Now, applying Theorem \ref{thm:smallt}, we get the above inequality for $t_0=0$, which ends the proof.
\end{proof}

Finally, we are prepared to prove the main result.
\begin{proof}[Proof of Theorem \ref{thm:main}]
The equivalent statement of the theorem is
$$u^*(t,x)\approx P_t^*u_0(x),\ \ \ t>0, x \in \RR^d$$
Theorems  \ref{thm:smallt} and \ref{thm:largex} imply that there exist $R>0$ and $t_0>0$ such that the required estimates hold whenever $t\in(0,t_0)$ or $|x|>R$. What has left is to consider $(t,|x|)\in[t_0,\infty)\times[0,R]$. Observe that by (\ref{eq:P>p}) and \eqref{est:p}, we have
\begin{equation}
c_1\|u_0\|_1 \le P_t^*u_0(x) \le c_2 \|u_0\|_1, \qquad (t,|x|)\in[t_0,\infty)\times[0,R],
\label{eq:thmeq1}
\end{equation}
for some constants $c_1,c_2>0$ ($c_1$ depends on $t_0$ and  $R$). To end the proof, we have to show that $u^*(t,x)\approx1$ for $t\ge t_0$ and $|x| \le R$. The upper bound comes from (\ref{est:pnormu*}). Next, under assumptions \eqref{con1} or \eqref{con2}, by \eqref{est:pnormu}, we have
\begin{align}
& \int_0^1 \int_{\RR^d} |b \cdot \nabla_w p(1-r^\alpha,x,rw)| [u^*(r^\alpha t,w)]^{q_0+1}dw\, dr \notag\\ 
& \le c_3 \int_0^{1/2} \int_{\RR^d}  [u^*(r^\alpha t,w)]^{1+q_0}dw\, dr  + \int_{1/2}^1 \int_{\RR^d} |\nabla_w p(1-r^\alpha,x,rw)| \|u_0\|_1^{1+q_0}dw\, dr \notag\\ 
& \le c_4 \|u_0\|_1^{1+q_0}
\label{eq:thmeq2}
\end{align}
for some constant $c_4=c_4(d,\alpha)>0$.

Now, let $\veps \in (0,1)$ and let $u_\veps(t,x)$ be the solution of (\ref{eq:problem}) with the initial condition $u_\veps(0,x)=\varepsilon\, u_0(x)$. Put $u^*_\veps(t,x)=t^{d/\alpha}u_\veps(t,t^{1/\alpha}x)$. Then, we have for every $t>0$, $\left\|u^*_\veps(t,\cdot)\right\|_\infty\leq\varepsilon \left\|u_0\right\|_1$ and $\left\|u^*_\veps(t,\cdot)\right\|_1\leq\varepsilon \left\|u_0\right\|_1$. Thus, by \eqref{eq:thmeq1} and \eqref{eq:thmeq2},
\begin{align*}
u^*_\veps(t,x)&\gtrsim\varepsilon P_t^*u_0(x) -\int_0^1 \int_{\RR^d}| b \cdot \nabla_w p(1-r^\alpha,x,rw) |[u^*_\varepsilon(r^\alpha t,w)]^{1+q_0}dw\, dr\\
&\ge \varepsilon c_1 \|u_0\|_1 - \varepsilon^{1+q_0} c_4\|u_0\|_1^{1+q_0},
\end{align*}
 for $t \ge t_0$ and $|x| \le R$. Taking $\varepsilon=\(\frac{c_1}{2c_3}\)^{1/q_0} \|u_0\|_1^{-1}$, we get $u^*_\veps(t,x)\ge \varepsilon c_5 \|u_0\|_1>0$. Since solutions of (\ref{eq:problem}) preserve the order of initial conditions (see  \cite{BKW1}, Lemma 3.1), we have $u^*(t,x)>u^*_\veps(t,x)$, and the proof is complete.
\end{proof}

\section{Asymptotic behaviour of solutions}
It is easy to see, that  
\begin{equation}\label{eq:P-Mp}
\lim_{t\rightarrow\infty}t^{n(1-1/p)/\alpha}\left\|\(P_tu_0\)(\cdot)-Mp(t,\cdot)\right\|_p=0
\end{equation} 
holds for every \mbox{$p\in[1,\infty]$} and $u_0\in L^1$. Applying this to (\ref{eq:u-P}), we obtain 
$$\lim_{t\rightarrow\infty}t^{n(1-1/p)/\alpha}\left\|u(t,\cdot)-Mp(t,\cdot)\right\|_p=0.$$ 
This form of the result is presented e.g. in   \cite{EZ}, where $\alpha=2$ is considered. It seems to be more useful then (\ref{eq:u-P}), since the function $p(t,x)$ is well known and does not depend on $u_0$. Such formulation is also a more natural counterpart of (\ref{eq:u-U}). Nevertheless, it may be concluded from Theorem \ref{thm:main}  that we have to employ the function $P_tu_0$ to describe the  behaviour of $u(t,x)$ more precisely. In the sequel, we discuss asymptotics of the quotient $u(t,x)/\(P_tu_0\)(x)$. We also give another improvement of (\ref{eq:u-P}). Some results are already provided in Section 4. In particular,  Proposition \ref{thm:smallt} is equivalent to the following equality. 
\begin{cor}\label{cor:t->0}
Under (\ref{con1}) or (\ref{con2}) we have
$$\lim_{t\rightarrow0}\left\|\frac{u(t,\cdot)}{\(P_t u_0\)(\cdot)}-1\right\|_\infty=0$$
\end{cor}
Theorem \ref{thm:largex} could be also reformulated in language of the function $u(t,x)$, but it would lose its clear form. Additionally, a stronger and clearer result, under condition (\ref{con2}), will be given  at the end of this section. Before that, we discuss the large time behaviour of the solution of (\ref{eq:problem}) with this condition.
\begin{prop}\label{cor:t->00}
Assume (\ref{con2}) holds. For every $0 < \gamma < (d(q-q_0) \land 1)/\alpha$, we have
$$\lim_{t\rightarrow\infty}t^{\gamma}\left\|1-\frac{u(t,\cdot)}{\( P_t u_0\)(\cdot)}\right\|_\infty=0.$$
\end{prop}
\begin{proof} 
There exists $\varepsilon>0$ such that $\gamma+\varepsilon< (d(q-q_0) \land 1)/\alpha$.  Additionally, using \eqref{con2}, we have 
$$\(P_su_0\)(z)  \lesssim  (s^{-d/\alpha}\|u_0\|_1) \land \|u_0\|_\infty.$$ 
Consequently, since  $q>q_0+\alpha(\gamma+\varepsilon)/d$, we get
$$[\(P_su_0\)(z)]^q \lesssim   [\(P_su_0\)(z)]	^{q_0+\frac{\alpha(\gamma+\varepsilon)}d} \lesssim  s^{-\frac{d}{\alpha}\(q_0+\frac{\alpha(\gamma+\varepsilon)}d\)} =  s^{-\gamma-\varepsilon-(\alpha-1)/\alpha}.$$
Then, by Theorem \ref{thm:main}, we obtain 
\begin{align}
&\left|\int_0^t \int_{\RR^d} b \cdot \nabla_w p(t-s,x,z)  [u(s,z)]^{q+1}\,dz\,ds\right| \notag\\
&\lesssim  \int_0^t \int_{\RR^d} (t-s)^{-1/\alpha} s^{-\gamma-\varepsilon-(\alpha-1)/\alpha} p(t-s,x,z)  \(P_su_0\)(z) \,dz\,ds \notag\\
&= t^{-\gamma-\varepsilon}\int_0^1 \int_{\RR^d} (1-r)^{-1/\alpha} r^{-\gamma-\varepsilon-(\alpha-1)/\alpha} p(t,x,w)u_0(w)\,dw\,ds \notag\\
& = c\, t^{-\gamma-\varepsilon} P_t u_0(x) \label{eq:t->00}.
\end{align}
The last integral is finite whenever $-\gamma-\varepsilon-(\alpha-1)/\alpha>-1$, which explains the importance of the assumption $\gamma+\varepsilon<1/\alpha$. Finally,  by  (\ref{eq:Duhamel2}) we arrive at 
$$ t^{\gamma}\left\|1-\frac{u(t,\cdot)}{\( P_t u_0\)(\cdot)}\right\|_\infty\lesssim t^{-\varepsilon},\ \ \ x\in\RR^d,\ t>0.$$
The proof is complete.
\end{proof}

\begin{rem} A result of that kind cannot be obtained in the case $q=q_0$, since (\ref{eq:u-U}) and (\ref{eq:P-Mp}) hold and 
$$\left\|1-\frac{U_M(t,\cdot)}{Mp(t,\cdot)}\right\|_\infty=\left\|1-\frac{U_M(1,\cdot)}{Mp(1,\cdot)}\right\|_\infty\neq0,$$
which follows from scaling properties of the functions  $p(t,x)$ and $U_M(t,x)$ (see \cite{BKW1}, Theorem 2.1). 
\end{rem}
The following result gives better large time asymptotics of $u(t,x)$ than one can obtain from \cite{BKW2} (see (\ref{eq:u-P})).
\begin{cor}
Assume (\ref{con2}) holds. For every $0 < \gamma < (d(q-q_0) \land 1)/\alpha$, we have
$$\lim_{t\rightarrow\infty}t^{\gamma + d(1-1/p)/\alpha}\left\|u(t,\cdot) - \(P_t u_0\) (\cdot) \right\|_p=0.$$
\end{cor}
\begin{proof}
Let, as in the proof of Proposition \ref{cor:t->00}, $\varepsilon>0$ such that $\gamma+\varepsilon< (d(q-q_0) \land 1)/\alpha$.  Then, \eqref{eq:t->00} gives us
\begin{align*}
|u(t,x) - \(P_t u_0\)(x)|  \lesssim  c t^{-\gamma-\varepsilon} \(P_tu_0\)(x)\,.
\end{align*}
By Young inequality, $\|\(P_tu_0\)(\cdot)\|_p \le \|p(t,\cdot)\|_p \|u_0\|_1 =c t^{-d(1-1/p)/\alpha}\|u\|_1$. Hence,
\begin{align*}
	\lim_{t\to\infty} t^{\gamma + d(1-1/p)/\alpha}\|u(t,\cdot) - \(P_t u_0\)(\cdot)\|_p =0.
\end{align*}
\end{proof}

Combining Theorem \ref{thm:largex} and Proposition \ref{cor:t->00}, we obtain the  uniform asymptotics for large $|x|$.
\begin{cor}
Under assumption (\ref{con2}), we have
$$
\lim_{|x|\rightarrow \infty} \sup_{t>0} \left|\frac{u(t,x)}{\(P_t u_0\)(x)}-1\right|=0.
$$
\end{cor}
\begin{proof}
Fix $\varepsilon>0$. In view of  Proposition \ref{cor:t->00}, there exists $t_0>0$ such that  
\begin{equation}\label{eq:u/P-1}
\left|\frac{u(t,x)}{\(P_t u_0\)(x)}-1\right|<\varepsilon,
\end{equation} whenever $t>t_0$. Furthermore, by Theorem \ref{thm:largex}, there is $R$ such that (\ref{eq:u/P-1}) holds if $|x|>R\,t^{1/\alpha}$. Consequently, (\ref{eq:u/P-1}) is true for $t>0$ and $|x|>R\,t_0^{1/\alpha}$, which ends the proof.
\end{proof}

\bibliography{burgers}
\bibliographystyle{plain}

\end{document}